\documentclass[12pt]{amsart}
\usepackage{a4wide,url,mathscinet,mathdots,bm,amsrefs,hyperref,comment,color}
\newcommand{\abs}[1]{\left|{#1}\right|}
\newcommand{\sm}[4]{{\bigl(\begin{smallmatrix}{#1}&{#2}\\{#3}&{#4}\end{smallmatrix}\bigr)}}

\newcommand{\GL}{\operatorname{GL}}

\newcommand{\R}{\mathbb{R}}
\newcommand{\C}{\mathbb{C}}
\newcommand{\Z}{\mathbb{Z}}

\newcommand{\one}{\mathbf{1}}
\newcommand{\sw}{\varsigma}
\newcommand{\JL}{\operatorname{JL}}
\newcommand{\Irr}{\operatorname{Irr}}
\newcommand{\Nrd}{\operatorname{Nrd}}
\newcommand{\Hom}{\operatorname{Hom}}
\newcommand{\rk}{n}
\newcommand{\lsg}{r}
\newcommand{\trs}{\mathbb{T}}

\newcommand{\pl}{\operatorname{pl}}

\newcommand{\un}{\kappa}
\newcommand{\dist}{\delta}
\newcommand{\OO}{\mathcal{O}}
\newcommand{\vol}{\operatorname{vol}}
\newcommand{\tor}{t}
\newcommand{\inv}{\iota}
\newcommand{\rest}{\lvert}

\newcommand{\prm}{\mathcal{P}}

\newtheorem{theorem}{Theorem}
\newtheorem{lemma}{Lemma}

\newtheorem{proposition}{Proposition}
\newtheorem{corollary}{Corollary}

\newtheorem{remark}{Remark}

\begin{document}

\title{On an inequality of Bushnell--Henniart for Rankin--Selberg conductors}
\author{Erez Lapid}
\address{Department of Mathematics, Weizmann Institute of Science, Rehovot 7610001, Israel}
\email{erez.m.lapid@gmail.com}

\begin{abstract}
We prove a division algebra analogue of an ultrametric inequality of Bushnell--Henniart for Rankin--Selberg conductors.
Under the Jacquet--Langlands correspondence, the two versions are equivalent.
\end{abstract}

\maketitle

\tableofcontents

\section{Introduction}

Let $F$ be a non-archimedean local field with residue field of size $q$. For any two smooth, generic, irreducible representations
$\pi_i$ of $\GL_{n_i}(F)$, $i=1,2$, Jacquet--Piatetski-Shapiro--Shalika defined \cite{MR701565} local factors
\[
L(s,\pi_1\times\pi_2),\ \ \epsilon(s,\pi_1\times\pi_2,\psi).
\]
The first is an inverse of a polynomials in $q^{-s}$ and the second is a monomial in $q^{-s}$ that
depends on a choice of a non-trivial character $\psi$ of $F$. We can write
\[
\epsilon(s,\pi_1\times\pi_2,\psi)=\epsilon(0,\pi_1\times\pi_2,\psi)q^{-(f(\pi_1\times\pi_2)+n_1n_2c(\psi))s}
\]
where $f(\pi_1\times\pi_2)\in\Z$ is the \emph{Rankin--Selberg conductor} and $c(\psi)$ is a certain integer depending
only on $\psi$. If $\psi$ is trivial on the ring of integers of $F$ but on no larger fractional ideal of $F$, then $c(\psi)=0$.
These local factors play an important role in the Plancherel decomposition of $\GL_n(F)$ \cites{MR729755, MR1070599}.
More precisely, if $\pi_1,\pi_2$ are square-integrable, then the density with respect to the Lebesgue measure
of the Plancherel measure  on $\GL_{n_1+n_2}(F)$ corresponding to discrete data
$(\GL_{n_1}(F)\times\GL_{n_2}(F),\pi_1\otimes\pi_2)$ is essentially given by
\[
\gamma(s,\pi_1\times\pi_2^\vee,\psi)\gamma(-s,\pi_1^\vee\times\pi_2,\psi^{-1})=
q^{f(\pi_1\times\pi_2^\vee)}\frac{L(1-s,\pi_1^\vee\times\pi_2)L(1+s,\pi_1\times\pi_2^\vee)}
{L(s,\pi_1\times\pi_2^\vee)L(-s,\pi_1^\vee\times\pi_2)}
\]
where
\[
\gamma(s,\pi_1\times\pi_2,\psi)=
\epsilon(s,\pi_1\times\pi_2,\psi)\frac{L(1-s,\pi_1^\vee\times\pi_2^\vee)}{L(s,\pi_1\times\pi_2)}.
\]
The factors $L(s,\pi_1\times\pi_2)$ are well-understood with regard to the Bernstein classification of square-integrable
representations in terms of cuspidal ones.
This is closely related to the fact that the reducibility points of the induced representations
$I(\pi_1\otimes\pi_2,s)$ are well understood.
On the other hand, the conductor $f(\pi_1\times\pi_2)$ is more mysterious, which is not surprising in view of
the local Langlands correspondence.
For instance, without the latter, it is not even clear why $f(\pi_1\times\pi_2)\ge0$.

The following interesting ultrametric property of the conductors was proved by Bushnell--Henniart.
\begin{theorem}[\cite{MR3664814}] \label{thm: ultra}
Let $\pi_i$ be irreducible cuspidal representations of $\GL_{n_i}(F)$, $i=1,2,3$. Then
\[
\frac{f(\pi_1\times\pi_3^\vee)}{n_1n_3}\le\max(\frac{f(\pi_1\times\pi_2^\vee)}{n_1n_2},\frac{f(\pi_2\times\pi_3^\vee)}{n_2n_3}).
\]
Consequently,
\[
\frac{f(\pi_1\times\pi_2^\vee)}{n_1n_2}\ge\max(\frac{f(\pi_1\times\pi_1^\vee)}{n_1^2},\frac{f(\pi_2\times\pi_2^\vee)}{n_2^2}).
\]
\end{theorem}
The result (and much more) is proved in \cites{MR1634095, MR3664814} using type theory and the classification of cuspidal
representations of $\GL_n$ \cite{MR1204652}.
The analogous result (under the local Langlands correspondence) for representations of the Weil--Deligne group was considered in
\cite{MR3711826}.
The purpose of this paper is to give an elementary short proof of an analogue of Theorem \ref{thm: ultra} for central division algebras
(see Proposition \ref{prop: CSdiv}).
Using the Jacquet--Langlands correspondence, one can then recover Theorem \ref{thm: ultra} in its original formulation (see \S\ref{sec: JL}).

In the course of the proof we will encounter an intriguing representation-theoretic interpretation of $f(\pi_1\times\pi_2)$ where $\pi_1,\pi_2$
are irreducible (finite-dimensional) representations of $D^*$ where $D$ is a central division algebra over $F$ (see \S\ref{sec: pl}).
In lieu of Rankin--Selberg theory, the conductor $f(\pi_1\times\pi_2)$ is defined, following Langlands, using the Plancherel measure with respect to $\GL_2(D)$.
The computation of the Plancherel measure is in the spirit of \cite{MR1606410}. However, it is simpler in our case and in particular,
we do not use the classification of irreducible representations of $D^*$.
On the other hand, without the Jacquet--Langlands correspondence, it is not clear that $f(\pi_1\times\pi_2)$ is an integer.

For the rest of the paper let $D$ be a central division algebra of degree $\rk$ over $F$.
Denote by $\abs{\cdot}$ the normalized absolute value on $F$ and by $\Nrd$ the reduced norm on $D$.
Let $\OO=\OO_D$ be the ring of integers of $D$ and let $\OO_D^*$ be the group of units of $\OO_D$.
We will take the Haar measure $d^*x$ on $D^*$ normalized by $\vol(\OO_D^*)=1$.
Denote by $v=v_D$ the valuation of $D$.
Thus, $\abs{\Nrd g}=q^{-v(g)}$ where $q=q_F$.
Let $\prm$ be the maximal ideal of $\OO$.
Let $U_m=1+\prm^m$, $m\ge1$ be the principal congruence subgroups of $D^*$, which are normal.
The group $D^*/\OO_D^*$ is infinite cyclic and the groups $D^*/F^*\OO_D^*$ and $\OO_D^*/U_1$ are cyclic of order $\rk$ and $q^{\rk}-1$
respectively.
Note that $\vol U_m=[\OO^*:U_m]^{-1}=\frac{q^{-m\rk}}{1-q^{-\rk}}$.
Denote by $\widehat{D^*}$ the set of equivalence classes of irreducible (finite-dimensional) unitary representations of $D^*$.\footnote{
The unitary restriction is made merely for convenience and is inconsequential.}
For any $\pi\in\widehat{D^*}$ denote by $d_\pi$ the dimension of $\pi$.

\section{A distance function} \label{sec: dist}

For any $\pi_1,\pi_2\in\widehat{D^*}$ define $\dist_{\pi_1,\pi_2}$ to be the smallest integer $m\ge0$ such that
$\pi_1\rest_{U_m}$ and $\pi_2\rest_{U_m}$ have a common irreducible constituent, i.e., $\Hom_{U_m}(\pi_1,\pi_2)\ne0$.
This definition makes sense for any group admitting a distinguished descending series of normal subgroups.
For the case of the Weil group see \cite{MR1411044}.
Clearly, $\dist_{\pi_1,\pi_2}=\dist_{\pi_2,\pi_1}$.

The following is a immediate consequence of Clifford theory.
\begin{lemma} \label{lem: cliff}
Let $\pi_1,\pi_2\in\widehat{D^*}$.
\begin{enumerate}
\item The following conditions are equivalent for $m\ge0$.
\begin{enumerate}
\item $m\ge\dist_{\pi_1,\pi_2}$
\item $\Hom_{U_m}(\pi_1,\pi_2)\ne0$.
\item $d_{\pi_2}\pi_1\rest_{U_m}\simeq d_{\pi_1}\pi_2\rest_{U_m}$ where $k\sigma$ denotes the direct sum of $k$ copies of $\sigma$.
\end{enumerate}
In this case,
\[
\frac{\dim\Hom_{U_m}(\pi_1,\pi_2)}{d_{\pi_1}d_{\pi_2}}=\frac{k}{d_{\sigma}^2}
\]
where $\sigma$ is an irreducible constituent of $\pi_j\rest_{U_m}$ and $k$ is the number of irreducible constituents of $\pi_j\rest_{U_m}$
up to equivalence (counted without multiplicity). (Neither $k$ nor $d_\sigma$ depends on either $\sigma$ or $j\in\{1,2\}$.)
\item For any $\pi_1,\pi_2,\pi_3\in\widehat{D^*}$
\[
\dist_{\pi_1,\pi_3}\le\max(\dist_{\pi_1,\pi_2},\dist_{\pi_2,\pi_3}).
\]
Moreover, if $m\ge\max(\dist_{\pi_1,\pi_2},\dist_{\pi_2,\pi_3})$ then
\[
\frac{\dim\Hom_{U_m}(\pi_1,\pi_3)}{d_{\pi_1}d_{\pi_3}}=\frac{\dim\Hom_{U_m}(\pi_1,\pi_2)}{d_{\pi_1}d_{\pi_2}}=\frac{\dim\Hom_{U_m}(\pi_2,\pi_3)}{d_{\pi_2}d_{\pi_3}}.
\]
\item $\dist_{\pi_1,\pi_2}=0$ if and only if $\pi_2$ is equivalent to a twist of $\pi_1$ by an unramified character.
\end{enumerate}
\end{lemma}

Let
\begin{gather*}
\inv_{\pi_1,\pi_2}=\int_{D^*}\max(v(g-1),0)\chi_{\pi_1}(g)\chi_{\pi_2}(g^{-1})\ d^*g
=\sum_{m=1}^\infty\int_{U_m}\chi_{\pi_1}(g)\chi_{\pi_2}(g^{-1})\ d^*g
\\=\sum_{m=1}^\infty\vol(U_m)\dim\Hom_{U_m}(\pi_1,\pi_2)=
\sum_{m=\max(1,\dist_{\pi_1,\pi_2})}^\infty\vol(U_m)\dim\Hom_{U_m}(\pi_1,\pi_2).
\end{gather*}
Clearly, $\inv_{\pi_1,\pi_2}=\inv_{\pi_2,\pi_1}$.
It is also clear that $\inv_{\pi_1,\pi_2}$ is unchanged if we twist $\pi_1$ or $\pi_2$ by an unramified character.
For simplicity we write $\inv_\pi=\inv_{\pi,\pi}$.

It follows from Lemma \ref{lem: cliff} that for any $\pi_1,\pi_2,\pi_3\in\widehat{D^*}$ we have
\begin{equation} \label{eq: ineq}
\frac{\inv_{\pi_1,\pi_3}}{d_{\pi_1}d_{\pi_3}}\ge\min(\frac{\inv_{\pi_1,\pi_2}}{d_{\pi_1}d_{\pi_2}},\frac{\inv_{\pi_2,\pi_3}}{d_{\pi_2}d_{\pi_3}}).
\end{equation}
(Note that the case where $\dist_{\pi_1,\pi_2}=0$ or $\dist_{\pi_2,\pi_3}=0$ is examined separately, but it is easy.)

Applying this inequality with $\pi_3=\pi_1$ we get
\[
\frac{\inv_{\pi_1}}{d_{\pi_1}^2}\ge\frac{\inv_{\pi_1,\pi_2}}{d_{\pi_1}d_{\pi_2}},
\]
and by symmetry
\[
\min(\frac{\inv_{\pi_1}}{d_{\pi_1}^2},\frac{\inv_{\pi_2}}{d_{\pi_2}^2})\ge\frac{\inv_{\pi_1,\pi_2}}{d_{\pi_1}d_{\pi_2}}.
\]
\section{A result of Olshanski}

Let $\pi_1,\pi_2\in\widehat{D^*}$.
Let $P$ (resp., $\bar P$) be the parabolic subgroup of $\GL_2(D)$ consisting of upper (resp., lower) triangular matrices.
Consider the family of induced representations
\[
I_P(s)=I_P(\pi_1\otimes\pi_2,s),\ \  I_{\bar P}(s)=I_{\bar P}(\pi_1\otimes\pi_2,s)
\]
on $\GL_2(D)$ parabolically induced from $\pi_1\abs{\Nrd}^{s/2}\otimes\pi_2\abs{\Nrd}^{-s/2}$
(resp., $\pi_1\abs{\Nrd}^{-s/2}\otimes\pi_2\abs{\Nrd}^{s/2}$). In both cases the induction is normalized.
Consider the intertwining operators
\begin{gather*}
M_P(s)=M_{P,\pi_1,\pi_2}(s):I_P(\pi_1\otimes\pi_2,s)\rightarrow I_{\bar P}(\pi_1\otimes\pi_2,-s)\\
M_{\bar P}(s)=I_{\bar P}(\pi_1\otimes\pi_2,s)\rightarrow I_P(\pi_1\otimes\pi_2,-s)
\end{gather*}
on $\GL_2(D)$ given by
\[
M_P(s)\varphi(g)=\int_D\varphi(\sm 1{}{x}1g)\ d^+x,\ \ M_{\bar P}(s)\varphi(g)=\int_D\varphi(\sm 1{x}{}1g)\ d^+x
\]
where we take the Haar measure $d^+x=\abs{\Nrd x}^{\rk}d^*x$ on $D$ (so that $\vol(\OO_D^*)=1$ with respect to this measure).
The integrals converge for $\Re s>0$ and admit meromorphic continuation to $\C$ as rational functions in $q^{-s}$.
If $\pi_1=\pi_2=\pi$ we will also write $M_P(\pi,s)$ and $M_{\bar P}(\pi,s)$.

For any $\pi\in\widehat{D^*}$ denote by $\tor=\tor(\pi)$ the size of the (cyclic) group of unramified characters of $D^*$
such that $\pi\otimes\omega\simeq\pi$. Clearly, $\tor\mid \rk$.

In order to analyze the poles of $M_P(s)$, we have to examine separately the case where $\dist_{\pi_1,\pi_2}>0$
and $\dist_{\pi_1,\pi_2}=0$. The latter reduced to the case $\pi_2=\pi_1$ upon twisting $\pi_1$ by an unramified character.
The following basic result is due to Olshanski.

\begin{proposition} \cite{MR0499010} \label{prop: olsh}
If $\dist_{\pi_1,\pi_2}>0$ then $M_P(s)$ is entire.
If $\pi_1=\pi_2=\pi$ then $(1-q^{-s\tor})M_P(s)$ is entire and
\[
\lim_{s\rightarrow 0}(1-q^{-s})M_P(s)=d_\pi^{-1}\cdot A_P
\]
where $A_P:I_P(0)\rightarrow I_{\bar P}(0)$ is the intertwining operator given by
\[
A_P\varphi(g)=S(\varphi(w_0^{-1}g))
\]
where $w_0=\sm{}{-1}{1}{}$ and $S:\pi\otimes\pi\rightarrow\pi\otimes\pi$ takes $u\otimes v$ to $v\otimes u$.
\end{proposition}

Note that $A_{\bar P}A_P$ is the identity.

\begin{proof}
For completeness, and since we will use a similar computation below, we will provide some details following the argument of Shahidi in
\cite{MR1748913}*{Proposition 5.1}.

For any $u\in\pi_1$, $v\in\pi_2$ and a Schwartz function $\Phi$ on $D$ let $\varphi_{u,v}\in I_P(s)$ be the section
supported in the big cell $P\bar P$ and given there by
\begin{equation} \label{eq: phi}
\varphi_{u,v}(\sm{g_1}{*}{}{g_2}\sm{1}{}{x}{1})=
\Phi(x)\abs{\frac{\Nrd g_1}{\Nrd g_2}}^{(s+\rk)/2}(\pi_1(g_1)u\otimes\pi_2(g_2)v),\ \ g_1,g_2\in D^*,\ x\in D.
\end{equation}

By a Lemma of Rallis \cite{MR1159430}*{Lemma 4.1}, it is enough to show that $[M_P(s)\varphi_{u,v}](w_0)$ is holomorphic at $s_0$
unless $\pi_2\simeq\pi_1\abs{\Nrd}^{s_0}$ and that if $\pi_2=\pi_1$, then $[M_P(s)\varphi_{u,v}](w_0)$ has a simple pole at $s=0$
with
\[
\lim_{s\rightarrow 0}(1-q^{-s})[M_P(s)\varphi_{u,v}](w_0)=d_\pi^{-1}\cdot [A_P\varphi_{u,v}](w_0)=d_\pi^{-1}\Phi(0)(v\otimes u).
\]

Using the relation
\[
\sm{1}{}{-x}{1}w_0=\sm{x^{-1}}{-1}{}{x}\sm{1}{}{x^{-1}}1,\ \ \ x\in D^*
\]
we write
\begin{gather*}
[M_P(s)\varphi_{u,v}](w_0)=\int_D\Phi(x^{-1})\abs{\Nrd(x)}^{-s-\rk}(\pi_1(x^{-1})u\otimes\pi_2(x)v)\ d^+x
\\=\int_{D^*}\Phi(x)\abs{\Nrd(x)}^s(\pi_1(x)u\otimes\pi_2(x^{-1})v)\ d^*x.
\end{gather*}
Hence, for any $u'\in\pi_1$, $v'\in\pi_2$
\begin{gather*}
([M_P(s)\varphi_{u,v}](w_0),u'\otimes v')=\int_{D^*}\Phi(x)\abs{\Nrd(x)}^s(\pi_1(x)u,u')(\pi_2(x^{-1})v,v')\ d^*x
\\=\int_{D^*/F^*}\int_{F^*}\omega_{\pi_1}\omega_{\pi_2}^{-1}(t)\abs{t}^{\rk s}\Phi(tx)\ d^*t\ \abs{\Nrd(x)}^s(\pi_1(x)u,u')(\pi_2(x^{-1})v,v')\ d^*x.
\end{gather*}
On $F^*$ we take the Haar measure normalized by $\vol(\OO_F^*)=1$ and we take the quotient measure on $D^*/F^*$.
The inner integral is a Tate integral (at $\rk s)$, whereas $D^*/F^*$ is compact and $\vol(D^*/F^*)=[D^*:\OO_D^*F^*]=\rk$. We have
\[
\lim_{s\rightarrow 0}(1-q^{-s})\int_{F^*}\omega_{\pi_1}\omega_{\pi_2}^{-1}(t)\abs{t}^s\Phi(tx)\ d^*t=\begin{cases}\Phi(0)&\text{if }\omega_{\pi_1}=\omega_{\pi_2},\\
0&\text{otherwise.}\end{cases}
\]
Thus, by the Schur orthogonality relations,
\[
\lim_{s\rightarrow 0}(1-q^{-s})[M_P(s)\varphi_{u,v}](w_0)=0
\]
unless $\pi_2\simeq\pi_1$ while if $\pi_1=\pi_2=\pi$, then
\[
\lim_{s\rightarrow 0}(1-q^{-s})([M_P(s)\varphi_{u,v}](w_0),u'\otimes v')=\Phi(0)d_\pi^{-1}(u,v')(v,u').
\]
The proposition follows.
\end{proof}

\section{Plancherel measure} \label{sec: pl}
The composition $M_{\bar P}(-s)M_P(s):I_P(\pi_1\otimes\pi_2,s)\rightarrow I_P(\pi_1\otimes\pi_2,s)$ is a scalar $\mu_{\pi_1,\pi_2}(s)^{-1}$, depending on $s$.
For simplicity we write $\mu_{\pi}(s)=\mu_{\pi,\pi}(s)$.

We now relate $\mu_{\pi_1,\pi_2}(s)$ to the objects introduced in section \ref{sec: dist}.
Once again, the case $\dist_{\pi_1,\pi_2}=0$ reduces to the case $\pi_2=\pi_1$.

\begin{proposition} \label{prop: mucomp}
Let $\pi_1,\pi_2\in\widehat{D^*}$.
\begin{enumerate}
\item If $\dist_{\pi_1,\pi_2}>0$ then
\[
\mu_{\pi_1,\pi_2}(s)^{-1}=\frac{\inv_{\pi_1,\pi_2}}{d_{\pi_1}d_{\pi_2}}.
\]
\item If $\pi_1=\pi_2=\pi$ then
\[
\mu_{\pi}(s)^{-1}=\frac{\inv_\pi+\un_\pi(s)}{d_\pi^2}
\]
where
\[
\un_\pi(s)=\frac{\tor^2}{(1-q^{\tor s})(1-q^{-\tor s})}.
\]
\end{enumerate}
\end{proposition}

Note that the proposition implies that
\[
\lim_{s\rightarrow 0}(1-q^{-s})^{-1}(1-q^s)^{-1}\mu_{\pi}(s)=d_\pi^2,
\]
in accordance with Proposition \ref{prop: olsh}.

\begin{proof}
Fix a Schwartz function $\Phi$ on $D$ and let $\varphi_{u,v}$ be as in \eqref{eq: phi}.
Using the relation
\[
\sm{1}{}{x}{1}\sm{1}{y}{}{1}=\sm{y(1+xy)^{-1}y^{-1}}{y}{}{1+xy}\sm{1}{}{(1+xy)^{-1}x}{1},\ \ x\in D, y\in D^*, x\ne y^{-1}
\]
we write $[M_P(s)\varphi_{u,v}]\sm{1}{y}{}{1}$ as
\begin{gather*}
\int_D\Phi((1+xy)^{-1}x)\abs{\Nrd(1+xy)}^{-s-\rk}(\pi_1(y(1+xy)^{-1}y^{-1})u\otimes\pi_2(1+xy)v)\ d^+x
\\=\abs{\Nrd y}^{-\rk}\int_{D^*}\Phi((1-x)y^{-1})\abs{\Nrd x}^s(\pi_1(yxy^{-1})u\otimes \pi_2(x)^{-1}v)\ d^*x.
\end{gather*}
Thus, fixing an orthonormal basis $u_i$ of $\pi_1$ and $v_j$ of $\pi_2$, we have
\[
\sum_{i,j}\big(M_P(s)\varphi_{u_i.v_j}(\sm{1}{h}{}{1}),u_i\otimes v_j\big)=
\abs{\Nrd h}^{-\rk}A(\Phi_h,s)
\]
where
\[
A(\Phi,s)=\int_{D^*}\chi_{\pi_1}(g)\chi_{\pi_2}(g^{-1})\Phi(1-g)\abs{\Nrd g}^s\ d^*g
\]
and $\Phi_h(x)=\Phi(xh^{-1})$. We view $A(\cdot,s)$ as a distribution on $D$ that converges for $\Re s>0$ and
admits meromorphic continuation. Thus,
\[
\sum_{i,j}(M_{\bar P}(-s)M_P(s)\varphi_{u_i.v_j},u_i\otimes v_j)=\int_{D^*}A(\Phi_h,s)\ d^*h
\]
which converges for $\Re s<0$.
As a distribution on $D$, the right-hand side is invariant under $\Phi\mapsto\Phi_h$ for any $h\in D^*$.
Hence, it must be proportional to $\Phi(0)$.
The constant of proportionality is $\mu_{\pi_1,\pi_2}(s)^{-1}d_{\pi_1}d_{\pi_2}$ since
$\sum_{i,j}\big(\varphi_{u_i.v_j}(e),u_i\otimes v_j\big)=d_{\pi_1}d_{\pi_2}\Phi(0)$.

Now take $\Phi$ to be the characteristic function of $\OO$. Let $h\in D^*$. For $\Re s>0$ we have
\[
A(\Phi_h,s)=\begin{cases}
\int_{1+h\OO_D}\chi_{\pi_1}(g)\chi_{\pi_2}(g^{-1})\ d^*g&\text{if }h\in\prm,\\
\int_{\abs{\Nrd(g)}\le\abs{\Nrd{h}}}\chi_{\pi_1}(g)\chi_{\pi_2}(g^{-1})\abs{\Nrd g}^s\ d^*g&\text{otherwise.}
\end{cases}
\]
Assume first that $\dist_{\pi_1,\pi_2}>0$. Then
\[
\int_{v(g)=k}\chi_{\pi_1}(g)\chi_{\pi_2}(g^{-1})\ d^*g=0
\]
for all $k\in\Z$. Hence, letting $m=v(h)$ we have
\[
A(\Phi_h,s)=\begin{cases}
\vol(U_m)\cdot\dim\Hom_{U_m}(\pi_1,\pi_2)&\text{if }h\in\prm\text{ (i.e., $m>0$)},\\
0&\text{otherwise.}
\end{cases}
\]
It follows that for $\Re s<0$ we have
\begin{gather*}
\mu_{\pi_1,\pi_2}^{-1}(s)d_{\pi_1}d_{\pi_2}=\int_{D^*}A(\Phi_h,s)\ d^*h=\int_{\abs{h}<1}A(\Phi_h,s)\ d^*h
\\=\sum_{m=1}^\infty\vol(U_m)\dim\Hom_{U_m}(\pi_1,\pi_2)=\inv_{\pi_1,\pi_2}
\end{gather*}
as claimed.

Assume now that $\pi_2=\pi_1=\pi$.
The restriction of $\pi$ to $\OO_D^*$ is multiplicity free of length $\tor$.
Moreover, for any integer $k$, $\chi_\pi$ vanishes on $\{g\in D^*:v(g)=k\}$ if and only if $k$ is not divisible by $\tor$.
Thus,
\[
\int_{v(g)=k}\abs{\chi_\pi(g)}^2\ d^*g=\begin{cases}\tor&\text{if }\tor\mid k,\\0&\text{otherwise.}\end{cases}
\]
Hence, letting $m=v(h)$ we have
\[
A(\Phi_h,s)=\begin{cases}
\vol(U_m)\dim\Hom_{U_m}(\pi,\pi)&\text{if }h\in\prm\text{ (i.e., $m>0$)},\\
\tor\frac{q^{\tor\lfloor\frac{-m}{\tor}\rfloor s}}{1-q^{-\tor s}}&\text{otherwise.}
\end{cases}
\]
It follows that for $\Re s<0$ we can write
\[
\mu_\pi(s)^{-1}d_{\pi}^2=\int_{D^*}A(\Phi_h,s)\ d^*h=
\int_{\abs{h}<1}A(\Phi_h,s)\ d^*h+\int_{\abs{h}\ge1}A(\Phi_h,s)\ d^*h
\]
where
\begin{gather*}
\int_{\abs{h}<1}A(\Phi_h,s)\ d^*h=
\sum_{m=1}^\infty\vol(U_m)\dim\Hom_{U_m}(\pi,\pi)=\inv_\pi
\end{gather*}
and
\[
\int_{\abs{h}\ge1}A(\Phi_h,s)\ d^*h=
\tor^2\sum_{m\ge0}\frac{q^{\tor ms}}{1-q^{-\tor s}}=
\frac{\tor^2}{(1-q^{\tor s})(1-q^{-\tor s})}
\]
as claimed.
\end{proof}

We specialize to the case $\pi_1=\pi_2=\pi$.

\begin{corollary} \label{cor: muiota}
Let $\pi\in\widehat{D^*}$. The representation $I_P(\pi,s)$, $s\in\R_{\ge0}$ is reducible at a unique point
$s=\lsg(\pi)$.\footnote{Of course, the uniqueness also follows from a general result on reductive groups.}
Moreover, $\lsg=\lsg(\pi)$ satisfies the following equivalent equations
\begin{subequations}
\begin{gather} \label{eq: bizar}
\inv_\pi=\tor^2\frac{q^{-\lsg\tor}}{(1-q^{-\lsg\tor})^2},\\
\label{eq: qr}
q^{\lsg\tor}+q^{-\lsg\tor}=2+\frac{\tor^2}{\inv_\pi},\\
\label{eq: plnch}
\mu_{\pi}^{-1}(s)d_\pi^2=\frac{\tor^2}{(1-q^{-\lsg\tor})^2}\frac{(1-q^{-\tor(s+\lsg)})(1-q^{\tor(s-\lsg)})}{(1-q^{\tor s})(1-q^{-\tor s})}.
\end{gather}
\end{subequations}
\end{corollary}

\begin{proof}
The equivalence of \eqref{eq: bizar} and \eqref{eq: qr} is straightforward, while the equivalence of \eqref{eq: bizar} and \eqref{eq: plnch}
follows from Proposition \ref{prop: mucomp}. The irreducibility of $I_P(\pi,0)$ follows from vanishing of $\mu_{\pi}(s)$ at $s=0$.
On the other hand, for $\Re s\ne0$, the reducibility points $I_P(\pi,s)$ coincide with the poles of $\mu_{\pi}(s)$.
Evidently, the equation \eqref{eq: qr} has a unique solution $\lsg>0$. The corollary follows.
\end{proof}


We shall say more about $\lsg$ in the next section.

We now define the conductor $f(\pi_1\times\pi_2^\vee)\in\R$ in terms of $\mu_{\pi_1,\pi_2}(s)$.
Set
\[
v_{\rk}=(1-q^{-\rk})q^{-\binom{\rk}2}.
\]
Suppose first that $\dist_{\pi_1,\pi_2}>0$. Then $f(\pi_1\times\pi_2^\vee)$ is defined by the equations
\begin{subequations}
\begin{gather} \label{eq: finv}
q^{-f(\pi_1\times\pi_2^\vee)}=\frac{v_{\rk}^2\inv_{\pi_1,\pi_2}}{d_{\pi_1}d_{\pi_2}},\\
\label{eq: muf}
\mu_{\pi_1,\pi_2}(s)=v_{\rk}^2q^{f(\pi_1\times\pi_2^\vee)},
\end{gather}
\end{subequations}
which are equivalent by Proposition \ref{prop: mucomp}.

If $\pi_1=\pi_2=\pi$ define $f(\pi\times\pi^\vee)\in\R$ by the equations
\begin{subequations}
\begin{gather} \label{eq: dpinv}
q^{-(f(\pi\times\pi^\vee)+\lsg\tor)}=\frac{v_{\rk}^2\inv_\pi}{d_\pi^2},\\
\label{eq: dpif}
d_{\pi}=v_{\rk}q^{f(\pi\times\pi^\vee)/2}\tor(1-q^{-\lsg\tor})^{-1},\\
\label{eq: mupid}
\mu_\pi(s)=v_{\rk}^2q^{f(\pi\times\pi^\vee)}\cdot\frac{(1-q^{-\tor s})(1-q^{\tor s})}{(1-q^{-\tor (s+\lsg)})(1-q^{-\tor (\lsg-s})}.
\end{gather}
\end{subequations}
Once again, the three equations are equivalent by Corollary \ref{cor: muiota}.

Note that the factor $v_{\rk}$ is introduced so that for $\pi=\one$ the one-dimensional identity representation we have $f(\pi\times\pi^\vee)=\rk^2-\rk$ (by \eqref{eq: dpif}).
However, it is not \emph{a priori} clear that $f(\pi_1\times\pi_2)$ is an integer for all $\pi_1,\pi_2\in\widehat{D^*}$.

In order to write \eqref{eq: finv} and \eqref{eq: dpinv} in a uniform way we introduce
\[
\tilde f(\pi_1\times\pi_2^\vee)=f(\pi_1\times\pi_2^\vee)+\lsg(\pi_1)\tor(\pi_1,\pi_2)
\]
where $\tor(\pi_1,\pi_2)$ is the number of unramified characters $\omega$ such that $\pi_2\simeq\pi_1\otimes\omega$.
Thus, $\tor(\pi_1,\pi_2)=\tor(\pi_1)=\tor(\pi_2)$ if $\dist_{\pi_1,\pi_2}=0$ and $\tor(\pi_1,\pi_2)=0$ otherwise.
We then have
\[
q^{-\tilde f(\pi_1\times\pi_2^\vee)}=\frac{v_{\rk}^2\inv_{\pi_1,\pi_2}}{d_{\pi_1}d_{\pi_2}}.
\]
From \eqref{eq: ineq} we infer:

\begin{proposition} \label{prop: CSdiv}
Let $\pi_1,\pi_2,\pi_3\in\widehat{D^*}$. Then,
\[
\tilde f(\pi_1\times\pi_3^\vee)\le\max(\tilde f(\pi_1\times\pi_2^\vee),\tilde f(\pi_2\times\pi_3^\vee)).
\]
Hence,
\[
\tilde f(\pi_1\times\pi_2^\vee)\ge\max(\tilde f(\pi_1\times\pi_1^\vee),\tilde f(\pi_2\times\pi_2^\vee)).
\]
\end{proposition}

\section{Jacquet--Langlands correspondence} \label{sec: JL}

Recall that we characterized the invariant $\lsg=\lsg(\pi)$ by the equivalent equations of Corollary \ref{cor: muiota}.
However, it is not {\it a priori} clear that $\lsg$ is an integer.

To that end we use the Jacquet--Langlands correspondence \cites{MR771672, MR700135, MR1737224} 
which is a natural bijection
\[
\JL:\widehat{D^*}\rightarrow\Irr_2\GL_{\rk}(F)
\]
where $\Irr_2\GL_{\rk}(F)$ denotes the set of square-integrable irreducible representations of $\GL_{\rk}(F)$, up to equivalence.
Denote by $\Irr_c\GL_{\rk}(F)\subset\Irr_2\GL_{\rk}(F)$ the subset of cuspidal representations.
Recall that any $\tau\in\Irr_2\GL_{\rk}(F)$ is of the form $\delta(\sigma,k)$ for a divisor $k$ of $\rk$ and
$\sigma\in\Irr_c\GL_{\rk/k}(F)$ (with $k$ and $\sigma$ uniquely determined by $\tau$) where $\delta(\sigma,k)$ denotes the (unique) irreducible subrepresentation
of the representation parabolically induced from $\sigma\abs{\cdot}^{x_1}\otimes\dots\otimes\sigma\abs{\cdot}^{x_k}$
where $x_1+x_k=0$ and $x_{i+1}-x_i=1$ for all $i=1,\dots,k-1$.
We have
\[
\JL(\pi)=\delta(\sigma,\lsg)
\]
for some $\sigma\in\Irr_c\GL_{\rk/\lsg}$ where $\lsg=\lsg(\pi)$.

One can also characterize $\lsg$ intrinsically in terms of $\pi$ without appealing to the Jacquet--Langlands correspondence.
Namely, the length of $\pi\rest_{U_1}$ is $\tor\frac{q^{\rk}-1}{q^{\lsg\tor}-1}$ \cite{MR2136524}.
Note that the proof in [ibid.] uses the classification of $\widehat{D^*}$ \cites{MR1635145, MR1166506}.
We are unaware of a proof of this fact (for \emph{some} integer $\lsg$ dividing $\rk/\tor$)
which does not use the classification of $\widehat{D^*}$.


\begin{remark}
The relation \eqref{eq: bizar} is remarkable. The right-hand side involves only the rather crude parameters $\tor$ and $\lsg$
which can be read off from the length of the restriction of $\pi$ to $\OO_D^*$ and to $U_1$.
On the other hand, $\inv_\pi$ ostensibly depends on the decomposition into irreducibles of the restriction of $\pi$ to $U_k$, $k\ge1$.
\end{remark}

\begin{remark}
In the case where the restriction of $\pi$ to $U_1$ is multiplicity free 
(which is satisfied for instance if $\tor=1$) the identity \eqref{eq: bizar} becomes
\[
\sum_{k\ge1}\int_{U_k}\abs{\chi_\pi(g)}^2\ d^*g=
\left(\int_{\OO_D^*}\abs{\chi_\pi(g)}^2\ d^*g+\int_{U_1}\abs{\chi_\pi(g)}^2\ d^*g\right)
\cdot\int_{U_1}\abs{\chi_\pi(g)}^2\ d^*g
\]
since
\[
\int_{\OO_D^*}\abs{\chi_\pi(g)}^2\ d^*g=\tor\text{ and }
\int_{U_1}\abs{\chi_\pi(g)}^2\ d^*g=\frac{\tor}{q^{\lsg\tor}-1}.
\]
\end{remark}

\begin{lemma} \label{lem: cp}
For any $\pi_1,\pi_2\in\widehat{D^*}$ we have
\[
f(\pi_1\times\pi_2^\vee)=f^{\GL}(\JL(\pi_1)\times\JL(\pi_2^\vee))
\]
where $f^{\GL}$ is the usual Rankin--Selberg conductor for the general linear group.
\end{lemma}

\begin{proof}
The Lemma is closely related to \cite{MR2131140}*{Theorem 7.2}, which relies on \cites{MR1951441, MR2024649}.
We follow the discussion of \cite{MR2131140}.
Let $\omega=\omega_{\pi_1}\omega_{\pi_2}$.
Let $\mu_{\pl,\omega}^G$ be the Plancherel measure on the set $\Irr_t(G,\omega)$ of irreducible tempered representations of $G$ with central character $\omega$. Similarly for $G'$.
Consider the torus
\[
\trs_{\pi_1,\pi_2}=\{s\in\C:\Re s=0\}/\{s:\pi_1\abs{\Nrd}^{s/2}\otimes\pi_2\abs{\Nrd}^{-s/2}\simeq\pi_1\otimes\pi_2\}.
\]
By \cite{MR1989693}, the restriction $\mu_{\pl,\pi_1\otimes\pi_2}^G$ of $\mu_{\pl,\omega}^G$ to $\{I_P(\pi_1\otimes\pi_2,s):s\in\trs_{\pi_1,\pi_2}\}$ is the pushforward of the measure on $\trs_{\pi_1,\pi_2}$ given by
\[
\mu_{\pi_1,\pi_2}(s)d_{\pi_1}d_{\pi_2}\ ds
\]
where $ds$ is a suitably normalized Lebesgue measure on $\trs_{\pi_1,\pi_2}$.

Let $\pi'_i=\JL(\pi_i)$, $i=1,2$.
Note that
\[
\trs_{\pi_1,\pi_2}=\{s\in\C:\Re s=0\}/\{s:\pi_1'\abs{\det}^{s/2}\otimes\pi'_2\abs{\det}^{-s/2}\simeq\pi'_1\otimes\pi'_2\}.
\]Let $P'$ be the parabolic subgroup of $G'$ consisting of block upper triangular matrices with blocks of size $\rk$
and $I_{P'}(\pi'_1\otimes\pi'_2,s)$ the representation of $G'$ parabolically induced from $\pi_1'\abs{\det}^{s/2}\otimes\pi_2'\abs{\det}^{-s/2}$.
Once more, the restriction $\mu_{\pl,\pi_1'\otimes\pi_2'}^{G'}$ of $\mu_{\pl,\omega}^{G'}$ to $\{I_{P'}(\pi'_1\otimes\pi'_2,s):s\in\trs_{\pi_1,\pi_2}\}$ is the pushforward of the measure on $\trs_{\pi_1,\pi_2}$ given by
\[
\mu'_{\pi_1',\pi_2'}(s)d'_{\pi'_1}d'_{\pi'_2}\ ds
\]
where $d'_{\pi'_i}$ is the formal degree and $\mu'_{\pi_1',\pi_2'}(s)$ is defined using intertwining operators as in \S\ref{sec: pl} .

By \cite{MR2131140}*{Theorem 7.2}, for suitable Haar measures on $G$ and $G'$, the measures 
$\mu_{\pl,\pi'_1\otimes\pi'_2}^{G'}$ is the pushforward of
$\mu_{\pl,\pi_1\otimes\pi_2}^G$ under the bijection $I_P(\pi_1\otimes\pi_2,s)\leftrightarrow I_{P'}(\pi'_1\otimes\pi'_2,s)$.
Moreover, $d'_{\pi'_i}=d_{\pi_i}$. Therefore 
\[
\mu'_{\pi_1',\pi_2'}(s)\sim\mu_{\pi_1,\pi_2}(s)
\]
where the proportionality constant is independent of the $\pi_i$'s. 
On the other hand, by \cite{MR729755}
\[
\mu'_{\pi_1',\pi_2'}(s)\sim
q^{f^{\GL}(\pi'_1\times{\pi'_2}^\vee)}\frac{L(1-s,{\pi'_1}^\vee\times\pi'_2)L(1+s,\pi'_1\times{\pi'_2}^\vee)}
{L(s,\pi'_1\times{\pi'_2}^\vee)L(-s,{\pi'_1}^\vee\times\pi'_2)}.
\]
By \cite{MR701565} we conclude
\[
\mu'_{\pi_1',\pi_2'}(s)\sim q^{f^{\GL}(\pi'_1\times{\pi_2'}^\vee)}\cdot\begin{cases}
\frac{(1-q^{-\tor(s+s_0)})(1-q^{\tor(s+s_0)})}{(1-q^{-\tor (s+s_0+\lsg)})(1-q^{-\tor (\lsg-s-s_0)})}&\text{if }\pi_1\simeq\pi_2\abs{\Nrd}^{s_0},\\
1&\text{otherwise.}\end{cases}
\]
Comparing with \eqref{eq: muf}, \eqref{eq: mupid} we conclude that
\[
f(\pi_1\times\pi_2^\vee)-f^{\GL}(\pi_1'\times{\pi_2'}^\vee)
\]
is a constant which is independent of $\pi_1$, $\pi_2$.
The lemma follows by examining the case where $\pi_1=\pi_2=\one$ (for which $\pi_1'=\pi'_2$ is the Steinberg representation of $\GL_{\rk}(F)$).
\end{proof}

As before, for any $\sigma\in\Irr_c\GL_{\rk}(F)$ we write $\tor(\sigma)$ for the size of the group of unramified
characters such that $\sigma\otimes\omega\simeq\sigma$.
Similarly, for any $\sigma_i\in\Irr_c\GL_{n_i}(F)$, $i=1,2$ we write $\tor(\sigma_1,\sigma_2)$ for the number of
unramified characters $\omega$ such that $\sigma_2\simeq\sigma_1\otimes\omega$.
Thus, $\tor(\sigma_1,\sigma_2)=\tor(\sigma_1)=\tor(\sigma_2)$ if $\tor(\sigma_1,\sigma_2)>0$ (which implies $n_2=n_1$).
We set
\[
\sw(\sigma_1,\sigma_2)=\frac{\tilde f^{\GL}(\sigma_1\times\sigma_2^\vee)}{n_1n_2}
\]
where
\[
\tilde f^{\GL}(\sigma_1\times\sigma_2^\vee)=f(\sigma_1\times\sigma_2^\vee)+\tor(\sigma_1,\sigma_2).
\]

We can now recover Theorem \ref{thm: ultra} (in its slightly stronger form, as in \cite{MR3664814}).

\begin{proposition} [\cite{MR3664814}] \label{prop: main}
Let $\sigma_i$ be irreducible cuspidal representations of $\GL_{n_i}(F)$, $i=1,2,3$. Then
\[
\sw(\sigma_1,\sigma_3)\le\max(\sw(\sigma_1,\sigma_2),\sw(\sigma_2,\sigma_3)).
\]
In particular
\[
\sw(\sigma_1,\sigma_2)\ge\max(\sw(\sigma_1,\sigma_1),\sw(\sigma_2,\sigma_2)).
\]
\end{proposition}

\begin{proof}
Let $\rk=n_1n_2n_3$, $\tau_i=\delta(\sigma_i,\frac{\rk}{n_i})\in\Irr_2\GL_{\rk}(F)$, $i=1,2,3$.
Fix a central division algebra $D$ of degree $\rk$ and let $\pi_i\in\widehat{D^*}$ be such that $\JL(\pi_i)=\tau_i$, $i=1,2,3$.
Then $\tor(\pi_i)=\tor(\delta_i)=\tor(\sigma_i)$ and by our assumption $\dist_{\pi_1,\pi_2},\dist_{\pi_1,\pi_3},\dist_{\pi_2,\pi_3}>0$.
By Lemma \ref{lem: cp}, for any $i,j$
\[
f(\pi_i\times\pi_j^\vee)=f^{\GL}(\tau_i\times\tau_j^\vee)=\frac{{\rk}^2}{n_in_j}\tilde f^{\GL}(\sigma_1\times\sigma_2^\vee)+
\tor(\sigma_i,\sigma_j)\frac{\rk}{n_i}(\frac{\rk}{n_i}-1)
\]
so that (since $\tor(\sigma_i,\sigma_j)=\tor(\pi_i,\pi_j)$)
\[
\tilde f(\pi_i\times\pi_j^\vee)=\frac{{\rk}^2}{n_in_j}\tilde f^{\GL}(\sigma_i\times\sigma_j^\vee).
\]
The first part now follows from Proposition \ref{prop: CSdiv}. The second part follows from the first part by taking $\sigma_3=\sigma_1$
and using symmetry.
\end{proof}

\subsection*{Acknowledgement}
I am grateful to Guy Henniart for useful discussions and suggestions.

\def\cprime{$'$}

\end{document}